\documentclass[reqno,a4paper, 11pt]{amsart}

\usepackage[a4paper=true,pdfpagelabels]{hyperref}
\usepackage{graphicx}

\usepackage[ansinew]{inputenc}
\usepackage{amsfonts,epsfig}
\usepackage{latexsym}
\usepackage{mathabx}
\usepackage{amsmath}
\usepackage{amssymb}

\usepackage{color}

\newtheorem{theorem}{Theorem}
\newtheorem{lemma}[theorem]{Lemma}

\newtheorem{proposition}[theorem]{Proposition}

\newtheorem{lettertheorem}{Theorem}
\newtheorem{letterlemma}[lettertheorem]{Lemma}

\theoremstyle{definition}

\theoremstyle{remark}
\newtheorem{remark}[theorem]{Remark}

\numberwithin{equation}{section}

\newcommand{\D}{\mathbb{D}}
\newcommand{\DD}{\widehat{\mathcal{D}}}
\newcommand{\Dd}{\widecheck{\mathcal{D}}}

\newcommand{\DDD}{\mathcal{D}}

\newcommand{\N}{\mathbb{N}}

\newcommand{\RR}{\mathbb{R}}

\newcommand{\C}{\mathbb{C}}

\renewcommand{\phi}{\varphi}

\newcommand{\whw}{\widehat{\omega}}
\newcommand{\veps}{\varepsilon}

\def\BMO{\mathord{\rm BMO}}

\def\d{\delta}           
     \def\om{\omega}      
              
                  \def\z{\zeta}

\def\LL{\mathcal{L}}

\renewcommand{\H}{\mathcal{H}}

\newenvironment{Prf}{\noindent{\emph{Proof of}}}
{\hfill$\Box$ }

\begin{document}
\title{Carleson measures for weighted Bergman--Zygmund spaces}

\today


\thanks{The first author is upported by NRF of Korea (NRF-2020R1F1A1A01048601),
the second by NRF of Korea (NRF-2022R1F1A1063305) and
the third by NRF of Korea(NRF-2019R1I1A3A01041943). The fourth author is partially supported by Wihuri Foundation. The fourth, fifth and sixth author is supported in part by Academy of Finland 356029.}

\author[H. R. Cho]{Hong Rae Cho}
\address{Department of Mathematics, Pusan National University, Busan 46241, Korea}
\email{chohr@pusan.ac.kr}

\author[H. Koo]{Hyungwoon Koo}
\address{Department of Mathematics, Korea University, Seoul 02841, Korea}
\email{koohw@korea.ac.kr}

\author[Y. J. Lee]{Young Joo Lee}
\address{Department of Mathematics, Chonnam National University, Gwangju 61186, Korea}
\email {leeyj@chonnam.ac.kr}

\author[A. Pennanen]{Atte Pennanen}
\address{University of Eastern Finland, P.O.Box 111, 80101 Joensuu, Finland}
\email{atte.pennanen@uef.fi}

\author[J. R\"atty\"a]{Jouni R\"atty\"a}
\address{University of Eastern Finland, P.O.Box 111, 80101 Joensuu, Finland}
\email{jouni.rattya@uef.fi}

\author[F. Wu]{Fanglei Wu}
\address{University of Eastern Finland, P.O.Box 111, 80101 Joensuu, Finland}
\email{fanglei.wu@uef.fi}

\begin{abstract}
For $0<p<\infty$, $\Psi:[0,\infty)\to(0,\infty)$ and a finite positive Borel measure $\mu$ on the unit disc $\D$, the Lebesgue--Zygmund space $L^p_{\mu,\Psi}$ consists of all measurable functions $f$ such that $\lVert f \rVert_{L_{\mu, \Psi}^{p}}^p =\int_{\D}|f|^p\Psi(|f|)\,d\mu< \infty$. For an integrable radial function $\om$ on $\D$, the corresponding weighted Bergman-Zygmund space $A_{\om, \Psi}^{p}$ is the set of all analytic functions in $L_{\mu, \Psi}^{p}$ with $d\mu=\om\,dA$.

The purpose of the paper is to characterize bounded (and compact) embeddings $A_{\omega,\Psi}^{p}\subset L_{\mu, \Phi}^{q}$, when $0<p\le q<\infty$, the functions $\Psi$ and $\Phi$ are essential monotonic, and $\Psi,\Phi,\om$ satisfy certain doubling properties. The tools developed on the way to the main results are applied to characterize bounded and compact integral operators acting from $A^p_{\om,\Psi}$ to $A^q_{\nu,\Phi}$, provided $\nu$ admits the same doubling property as $\omega$.
\end{abstract}

\maketitle

\section{Introduction and main results}

For $0<p<\infty$, $\Psi: [0, \infty) \to (0,\infty)$ and a finite positive Borel measure $\mu$, the Lebesgue-Zygmund space $L_{\mu, \Psi}^{p}$ consists of all measurable functions $f$ in the unit disc $\D=\{z\in\C:|z|<1\}$ such that
	$$
	\lVert f \rVert_{L_{\mu, \Psi}^{p}} = \left(\int_{\D}|f|^p\Psi(|f|)\,d\mu\right)^{\frac1p} < \infty.
	$$
In this study we assume that the inducing function $\Psi$ satisfies certain hypotheses. We say that it belongs to the class $\mathcal{L}$ if it is essentially monotonic and satisfies the property
	\begin{equation}\label{doubling}
	\Psi(x) \asymp \Psi(x^2),\quad 0\leq x<\infty.
	\end{equation}
Here, $\Psi$ is essentially increasing if there exists an absolute constant $C=C(\Psi)>0$ such that $\Psi(x)\le C\Psi(y)$ for all $x\le y$. The essentially decreasing functions are defined in an analogous manner. Besides constants, for each $\alpha\in\RR$, the logarithmic function $x\mapsto (\log(e+x))^{\alpha}$ is a typical member of $\LL$. In fact, we will show in Section~\ref{Sec2} that each essentially increasing function in $\LL$ cannot grow faster than a certain logarithmic function while the decay of each essentially decreasing function in $\LL$ is restricted in a similar manner. It is worth mentioning that either requirement for a function to belong to $\mathcal{L}$ does not imply the other, see Section~\ref{Sec2} for details.
	
If $d\mu=\om\,dA$, where $dA(z)=\frac{dx\,dy}{\pi}$ is the element of the normalized Lebesgue area measure on $\D$, then we write $L_{\mu, \Psi}^{p}=L_{\om, \Psi}^{p}$. For $\Psi(x)=\left(\log(e+x)\right)^{\beta}$ with $0<\beta<\infty$, the corresponding Lebesgue-Zygmund space becomes the classical Zygmund space, see~\cite{BR, BR2, BergZyg} and references therein for further information. The Lebesgue-Zygmund space $L_{\mu, \Psi}^{p}$ is not necessarily a normed space even in the case $p\geq1$, because $\|\cdot\|_{L^p_{\mu,\Psi}}$ does not necessarily satisfy the absolute homogeneity. However, it can be proved that it is a quasinorm by standard methods.

Let $\H(\D)$ denote the space of analytic functions in $\D$. An integrable function $\om:\D\to [0,\infty)$ is called a weight. The weighted Bergman-Zygmund space $A_{\om, \Psi}^{p}$ consists of all analytic functions in $L_{\om, \Psi}^{p}$, that is, $A_{\om, \Psi}^{p}=L_{\om, \Psi}^{p}\cap\H(\D)$. The definition of the weighted Bergman-Zygmund spaces seems similar to that of certain weighted Bergman-Orlicz spaces. More precisely, if the function $t\mapsto t^p\Psi(t)$ happens to be non-decreasing and convex, then the corresponding Bergman-Orlicz space has the same elements as the Bergman-Zygmund space induced by $\Psi$ and $\om$. However, it is easy to see that their quasi-norms are not comparable. For more information on the weighted Bergman-Orlicz spaces, see \cite{Orlicz1} and \cite{Orlicz2}.

In this paper, we are interested in the Bergman-Zygmund spaces induced by radial weights satisfying certain doubling properties. To be precise, some more notation is needed. We say that $\om$ is radial if $\om(z)=\om(|z|)$ for all $z\in\D$, and we write $\whw(z)=\int_{|z|}^1\om(s)\,ds$ for short. We assume $\widehat{\om}>0$ on $\D$ for otherwise $A_{\om, \Psi}^{p}=\H(\D)$.
A radial weight $\om$ belongs to the class~$\DD$ if there exists a constant $C=C(\om)\ge1$ such that the tail integral $\widehat{\om}$ satisfies the doubling condition $\widehat{\om}(r)\le C\widehat{\om}(\frac{1+r}{2})$ for all $0\le r<1$. Moreover, if there exist constants $K=K(\om)>1$ and $C=C(\om)>1$ such that $\widehat{\om}(r)\ge C\widehat{\om}\left(1-\frac{1-r}{K}\right)$ for all $0\le r<1$, then we write $\om\in\Dd$. Finally, the intersection $\DD\cap\Dd$ is denoted by $\DDD$. The definitions of
these classes of weights are of geometric in nature, and the classes themselves arise
naturally in the operator theory related to Bergman spaces, see for example \cite{PelRat2020, PRV}.

For $\Phi,\Psi\in\LL$, we say that $\mu$ is a $(q,\Phi)$-Carleson measure for $A_{\omega,\Psi}^{p}$ if the identity mapping $I: A_{\omega,\Psi}^{p} \to L_{\mu, \Phi}^{q}$ is bounded. Here an operator is called bounded if it maps bounded sets into bounded sets, and further, it is compact if it sends a neighborhood of the origin onto a relatively compact set in the target space.

In the paper, we will characterize the $(q,\Phi)$-Carleson measures for $A_{\omega,\Psi}^{p}$, provided $0<p\le q<\infty$, $\om\in\DDD$ and $\Phi,\Psi\in\LL$. To give the precise statement some basic definitions are in order. The Carleson square associated with $a\in\D\setminus\{0\}$ is the polar rectangle
	$$
	S(a)=\left\{z\in\D:\,|z|\ge|a|,\,|\arg a-\arg z|<\frac12(1-|a|)\right\}.
	$$
Further, we set $S(0)=\D$ for convenience, and let $\omega(S)=\int_S \omega dA$ for measurable $S\subset \D$.
Our main result reads as follows.

\begin{theorem}\label{mainresult}
Let $0<p\le q<\infty$, $\omega \in \mathcal{D}$ and $\Phi,\Psi\in\LL$. Then the following statements hold:
    \begin{itemize}
        \item [(i)]$\mu$ is a $(q,\Phi)$-Carleson measure for $A_{\omega,\Psi}^{p}$ if and only if
    \begin{equation}\label{bdd}
    \sup_{a\in \D}\frac{\mu(S(a))\Phi\left(\frac{1}{\omega(S(a)))}\right)}{\left(\omega(S(a))\Psi\left(\frac{1}{\omega(S(a))}\right)\right)^{\frac{q}{p}}} < \infty;
   \end{equation}
     \item [(ii)]$I:A_{\omega,\Psi}^{p}\to L^q_{\mu,\Phi}$ is compact if and only if
    \begin{equation}\label{cpt}
\lim_{|a|\to1^-}\frac{\mu(S(a))\Phi\left(\frac{1}{\omega(S(a))}\right)}{\left(\omega(S(a))\Psi\left(\frac{1}{\omega(S(a))}\right)\right)^{\frac{q}{p}}}=0.
    \end{equation}
    \end{itemize}
\end{theorem}

In the statement of Theorem~\ref{mainresult} we may replace the Carleson square $S(a)$ by the pseudohyperbolic disc $\Delta(a,r)=\{z\in \D:|(a-z)/(1-\overline{a}z)|<r\}$ centered at $a\in\D$ and of radius $r\in(0,1)$. This is due to the hypothesis $\om\in\DDD$ and standard arguments, see Section~\ref{Sec3} for details.
This phenomenon was first noticed in \cite{Lu1985} for the Bergman space and  also appears in the case of $q$-Carleson measures for Bergman space $A^p_\om$ induced by $\om\in\DDD$ by \cite{BinRattya,LRW2021}. Theorem~\ref{mainresult} reduces to \cite[Theorem~2]{BinRattya} when $\Phi$ and $\Psi$ are constants, and the choices $\Phi(x)=\left(\log(e+x)\right)^\gamma$ and $\Psi(x)=\left(\log(e+x)\right)^\beta$ give \cite[Theorem~C]{BergZyg} when the weight involved is $\om(z)=(1-|z|)^\alpha$.

As a byproduct of the proof of Theorem \ref{mainresult} and the theorem itself we deduce that for each $\om\in\DDD$ and $\Psi\in\mathcal{L}$, the weight $W=\Psi\left(\frac{1}{1-|\cdot|}\right)\om$ belongs to $\DDD$ and $A^p_{\om,\Psi}$ equals $A^p_W$ as a set. This is the main content of Proposition~\ref{setequiv} which is proved at the end of the paper. However, it should be underlined here that even though the spaces have the same elements, the quasinorms are not in general comparable.

En route to the proof of Theorem~\ref{mainresult} we develop several tools that allow us to study integral operators on weighted Bergman--Zygmund spaces. The integral operator $T_g$, induced by a $g\in\H(\D)$, is defined by
	\[
	T_g(f)(z)=\int_{0}^{z}f(\z)g'(\z)\,d\z, \quad z\in \D.
	\]
This operator was initially considered by Pommerenke~\cite{Pomme}, and it began to be extensively studied after the appearance of the works by Aleman, Cima and Siskakis \cite{AC,AS}. In the paper, we will characterize the boundedness and compactness of $T_g: A^p_{\om,\Psi} \to A^q_{\nu,\Phi}$ in terms of neat conditions on $g$, provided $0<p\leq q<\infty$ and $\om,\nu\in\DDD$. Our results in this direction are stated as follows.

\begin{theorem}\label{Tgbounded}
    Let $0<p\leq q<\infty$, $\om,\nu \in \DDD$ and $\Psi,\Phi \in \LL$, $g\in\H(\D)$. Then the following statements hold:
\begin{itemize}
\item[(i)]$T_g: A^p_{\om,\Psi} \to A^q_{\nu,\Phi}$ is bounded if and only if
    \begin{equation}\label{Tgcondition}
        \sup_{a \in \D}\frac{|g'(a)|(1-|a|)\left(\Phi\left(\frac{1}{\om(S(a))}\right)\nu(S(a))\right)^\frac{1}{q}}
    {\left(\Psi\left(\frac{1}{\om(S(a))}\right)\om(S(a))\right)^\frac{1}{p}}<\infty;
    \end{equation}
    \item[(ii)] $T_g: A^p_{\om,\Psi} \to A^q_{\nu,\Phi}$ is compact if and only if
    \begin{equation}\label{Tgcondition2}
        \lim_{|a|\to 1^-}\frac{|g'(a)|(1-|a|)\left(\Phi\left(\frac{1}{\om(S(a))}\right)\nu(S(a))\right)^\frac{1}{q}}
    {\left(\Psi\left(\frac{1}{\om(S(a))}\right)\om(S(a))\right)^\frac{1}{p}}=0.
    \end{equation}
\end{itemize}
\end{theorem}

We make couple of observations on Theorem~\ref{Tgbounded}. First, as in the case of Carleson measures and because of the same reason, pseudohyperbolic discs can be substituted for the Carleson squares. Second, the special case $p=q$ and $\om=\nu$ in Theorem~\ref{Tgbounded} shows that $T_g: A^p_\om\to A^q_\om$ is bounded (resp. compact) if and only if $g$ belongs to the (resp. little) Bloch space \cite[Theorem~6.4]{PelSum14}.

We finish the introduction with a couple of words about the notation used in this paper. The letter $C=C(\cdot)$ will denote an absolute constant whose value depends on the parameters indicated in the parenthesis and may change from one occurrence to another. If there exists a constant
$C=C(\cdot)>0$ such that $a\le Cb$, then we write either $a\lesssim b$ or $b\gtrsim a$. In particular, if $a\lesssim b$ and
$a\gtrsim b$, then we denote $a\asymp b$ and say that $a$ and $b$ are comparable.

\section{Class $\LL$}\label{Sec2}

Recall that a function $\Psi: [0, \infty) \to (0,\infty)$ belongs to the class $\mathcal{L}$ if it is essentially monotonic and satisfies the property $\Psi(x)\asymp\Psi(x^2)$ for all $0\le x<\infty$.
It is worth mentioning that either requirement for a function to belong to $\mathcal{L}$ does not imply the other. Namely, the essential monotonicity obviously does not imply the condition \eqref{doubling} while a non essentially monotonic function  satisfying \eqref{doubling} can be constructed as follows. Choose an increasing and unbounded function $f$ and a decreasing function $g$ such that $f(x) \asymp f(x^2)$ and $g(x)\asymp g(x^2)$ for all $[0,\infty)$. Fix $x_1$ and $y_1$ such that $x_1^2>y_1>x_1>1$, and set $x_{n+1}=x_{n}^2$ and $y_{n+1}=y_{n}^2$ for all $n\in\N$. Consider $\Psi$ defined by
	$$
	\Psi(x)=\begin{cases}
    g(x), \quad x \in \bigcup_{n=1}^\infty(x_n,y_n), \\
    f(x), \quad x \in [0,\infty)\setminus \bigcup_{n=1}^\infty(x_n,y_n).
	\end{cases}
	$$
Then, $\Psi(x) \asymp \Psi(x^2)$ for all $x\in[0,\infty)$, but $\Psi$ is not essentially monotonic. Natural choices of $f$ and $g$ are adequate powers of logarithmic functions.

In this section we provide some basic properties of functions in the class $\LL$ that will be later on used in the proof of the main theorem. We begin with pointwise growth/decay estimates.

\begin{lemma}\label{GrowthPsi}
Let $\Psi\in\LL$. Then there exist constants $c_1=c_1(\Psi), C_1=C_1(\Psi)>0$ and $c_2=c_2(\Psi), C_2=C_2(\Psi)\in \RR$ such that
	$$
	c_1(\log(e+x))^{c_2}\leq \Psi(x) \leq C_1(\log(e+x))^{C_2},\quad 0\le x<\infty.
	$$
\end{lemma}

\begin{proof}
Let
	$$
	a=\inf_{0\le x\le4} \Psi(x) \quad\mbox{and}\quad
	b=\sup_{0\le x\le4} \Psi(x).
	$$
Since $\Psi\in\LL$, both $a$ and $b$ are strictly positive and finite, and thus the statement is valid for all $0\le x\le 4$.
Let  $4<x<\infty$, and choose $n\in\N$ such that $2^{2^n}\le x<2^{2^{n+1}}$. Then, by \eqref{doubling}, there exist constants
$\widetilde c_2, \widetilde C_2>0$ such that $\widetilde c_2^n a \le \Psi(x) \le \widetilde C_2^n b$. Moreover, since $n\le\log_2\log_2x<n+1$, there exist constants {$c_2, C_2 \in \RR$} such that
	$$
	\widetilde c_2^n \asymp (\log(e+x))^{c_2},
	\quad\textrm{and}\quad
	\widetilde C_2^n \asymp (\log(e+x))^{C_2}.
	$$
The assertion follows from the estimates established.
\end{proof}

The functions in the class $\LL$ admit the following pointwise estimates which are useful for our purposes.

\begin{proposition}\label{properties}
Let $0<p<\infty$ and $\Phi,\Psi\in\LL$. Then the following statements hold:
			\begin{itemize}
      \item [(i)]$\Psi(x)\asymp \Psi(y)$, when $x \asymp y$;
      \item[(ii)]$\Psi(x)\asymp\Psi(x^p)$ for all $0\le x<\infty$;
      \item[(iii)]$\Psi(x/\Phi(x))\asymp\Psi(x)$ for all $0\le x<\infty$.
			\end{itemize}
\end{proposition}

\begin{proof}
It is easy to see that (i) and (ii) hold, while (iii) follows from (ii) and Lemma~\ref{GrowthPsi}.
\end{proof}

As an immediate corollary to Lemma~\ref{GrowthPsi} we obtain another handy result. Its essential content is that the product of $\Psi\in\LL$ and any power of the identity function results an essentially increasing function. Observe that this statement can not be proved by using only the essential monotonicity of $\Psi$. Namely, for an example the function defined by $\Psi(x)=2^{-2^n}$ for $n\le x<n+1$ is decreasing but $x\mapsto x^p \Psi(x)$ is not essentially increasing.

\begin{lemma}\label{ThetaIncreasing}
Let $0<p,\beta<\infty$ and $\Psi \in \LL$. Then $\Theta:x\mapsto x^p \Psi(x)$ is essentially increasing on $[0,\infty)$, and $\theta: x\mapsto (1-x)^{\beta}\Psi\left(\frac{1}{1-x}\right)$ is essentially decreasing on $[0,1)$.
\end{lemma}

\begin{proof}
We prove the first statement only, the latter one can be proved similarly. So we will show that there exists $C=C(\Psi,p)>0$ such that $\Theta(x)\le C \Theta(y)$ for all $0\le x\le y<\infty$. For $0\le x\le2$, this follows trivially by Lemma~\ref{GrowthPsi}. Moreover, if $2<x<\infty$ and $x\le y<x^2$, then the inequality is a direct consequence of Proposition~\ref{properties}(ii). Let now $2<x<\infty$, $x^2\le y$ and write $y=x^t$.
Then, for each $\alpha>0$ there exists $C_1=C_1(\alpha)>0$ such that
	\begin{align*}
	\left(\log(e+x)\right) \left(\log(e+y)\right)
	\le   t \left(\log(e+x)\right)^2 \le C_1 x^{\alpha (t-1)}
	= C_1 \left({y\over x}\right)^\alpha.
	\end{align*}
By Lemma \ref{GrowthPsi} there are $C_2=C_2(\Psi),\beta=\beta(\Psi)>0$ such that
	\begin{align*}
	{\Psi(y)\over \Psi(x)}\ge {C_2 \over  \left(\log(e+x)\right)^\beta \left(\log(e+y)\right)^\beta}.
	\end{align*}
Therefore, by choosing $\alpha>0$ sufficiently small such that $\alpha\beta<p$, we deduce
	\begin{align*}
	{y^p\Psi(y)\over x^p\Psi(x)}\ge   \left({y\over x}\right)^p {C_2 \over  \left(\log(e+x)\right)^\beta \left(\log(e+y)\right)^\beta}>C
	\end{align*}
for some $C=C(\Psi)>0$.
\end{proof}

\section{Proof of the main theorems}\label{Sec3}

For $f\in\H(\D)$ and $0<r<1$, set
	$$
	M_p(r,f)=\left(\frac{1}{2\pi}\int_0^{2\pi}|f(re^{i\theta})|^p\,d\theta\right)^{\frac1p},\quad 0<p<\infty,
	$$
and $M_\infty(r,f)=\sup_{|z|=r}|f(z)|$. The following lemma gives a growth estimate for functions in the weighted Bergman-Zygmund space. We will use this estimate in the proof of Theorem~\ref{mainresult}.

\begin{lemma}\label{Growthf}
Let $0<p<\infty$, $\omega \in \widehat{\mathcal{D}}$ and $\Psi\in\LL$. Then, for each fixed $0<\veps<p$,
    \begin{equation}\label{growth}
    |f(z)|^{p-\veps} \lesssim \frac{\lVert f \rVert^p_{A_{\omega,\Psi}^{p}}+1}{\om(S(z))}, \quad f\in A_{\omega,\Psi}^{p},\quad z\in\D.
    \end{equation}
\end{lemma}

\begin{proof}
It follows from \cite[Lemma~1.3]{P1} that
		$$
		M_{\infty}(r,f)\lesssim \frac{M_{p}(\frac{1+r}{2},f)}{(1-r)^{\frac{1}{p}}},\quad f \in \mathcal{H}(\D),\quad 0\le r<1.
		$$
This estimate together with Lemma~\ref{ThetaIncreasing} and the hypotheses $\Psi\in\LL$ and $\om\in\DD$ yields
		\begin{align*}
    |f(z)|^{p-\veps}\omega(S(z))
		&\lesssim M_{\infty}^{p-\veps}(|z|,f)(1-|z|)\widehat{\omega}(z)
		\lesssim M_{p-\veps}^{p-\veps}\left(\frac{1+|z|}{2}\right)\widehat{\omega}\left(\frac{1+|z|}{2}\right) \\
    &\lesssim\int_{\frac{1+|z|}{2}}^{1}M_{p-\veps}^{p-\veps}(r,f) \omega(r)r\,dr
		\le\int_\D |f(\zeta)|^{p-\veps}\om(\zeta)\,dA(\zeta) \\
    &=\int_{|f|\leq 1}|f(\zeta)|^{p-\veps}\om(\zeta)\,dA(\zeta)
		+\int_{|f|>1}|f(\zeta)|^{p-\veps}\frac{\Psi(|f(\zeta)|)}{\Psi(|f(\zeta)|)}\om(\zeta)\,dA(\zeta)\\
    &\lesssim 1 + \int_{|f|>1}|f(\zeta)|^{p}\Psi(|f(\zeta)|)\om(\zeta)\,dA(\zeta)
		\le1 + \lVert f \rVert^p_{A_{\omega,\Psi}^{p}}.
		\end{align*}
For the case where $\Psi$ is essentially increasing, we not need to use Lemma~\ref{ThetaIncreasing}, and we can actually choose $\veps=0$ and remove the added constant $1$.
\end{proof}

Next, we show that the quantity that defines the weighted Lebesgue-Zygmund space satisfies the quasi triangular inequality and the scalar multiplication is continuous. Observe that we do not require absolute homogeneity for a set to be a quasinormed vector space.

\begin{proposition}\label{quasinorm}
Let $0<p<\infty$ and $\Psi\in\LL$, and let $\mu$ be a finite positive Borel measure on $\D$. Then the space $L_{\mu,\Psi}^{p}$ is a quasinormed topological vector space with the quasinorm
    $$
    \lVert f \rVert_{L_{\mu,\Psi}^{p}} = \left(\int_{\D}|f|^p\Psi(|f|)d\mu\right)^{\frac1p}.
    $$
\end{proposition}

\begin{proof}
To prove that $\lVert \, \cdotp \rVert_{L_{\mu,\Psi}^{p}}$ is a quasinorm, we only need to show that
	$$
	\lVert f+g  \rVert_{L_{\mu,\Psi}^{p}}
	\lesssim\lVert f \rVert_{L_{\mu,\Psi}^{p}}+\lVert g \rVert_{L_{\mu,\Psi}^{p}}.
	$$
By symmetry, we may assume that $|g|\le|f|$. Then, by Lemma~\ref{ThetaIncreasing} and Proposition~\ref{properties}, we have
    \begin{align*}
    |f+g|^p\Psi(|f+g|) \lesssim 2^p|f|^p\Psi(|f|) \lesssim |f|^p\Psi(|f|)+|g|^p\Psi(|g|).
    \end{align*}
This clearly gives us
    \begin{align*}
     \lVert f+g \rVert_{L_{\mu,\Psi}^{p}} &= \left(\int_{\D}|f+g|^p\Psi(|f+g|)\,d\mu\right)^{\frac1p}\\
        &\leq (2C)^{\frac1p}\left(\left(\int_{\D}|f|^p\Psi(|f|)\,d\mu\right)^{\frac1p}+\left(\int_{\D}|g|^p\Psi(|g|)\,d\mu\right)^{\frac1p}\right)\\
        &\leq (2C)^{\frac1p}\left(\lVert f \rVert_{L_{\mu,\Psi}^{p}}+\lVert g \rVert_{L_{\mu,\Psi}^{p}} \right)
    \end{align*}
for some constant $C=C(p,\Psi)>0$. Thus $\lVert \, \cdotp \rVert_{L_{\mu,\Psi}^{p}}$ satisfies the quasi-triangle inequality. This trivially implies that addition is continuous, so it remains to show that the scalar multiplication is continuous. Assume first that $\Psi$ is essentially increasing. When $|c|<1$, we have $\Psi(|cf|)\lesssim \Psi(|f|)$. If $|c|>1$, choose $k \in \N\cup\{0\}$ such that $2^k\leq |c| < 2^{k+1}$. Then Proposition~\ref{properties}(i) implies
    $$
    \Psi(|cf|)\lesssim \Psi(2^{k+1}|f|) \leq C^k \Psi(|f|) < C^{\log_2 |c|}\Psi(|f|),
    $$
which gives us
    $$
     \lVert cf \rVert^p_{L_{\mu,\Psi}^{p}}\lesssim |c|^p(1+C^{\log_2 |c|})\lVert f \rVert^p_{L_{\mu,\Psi}^{p}}.
     $$
     Assume next that $\Psi$ is essentially decreasing. When $|c|\geq 1$, we clearly have $\Psi(|cf|)\lesssim \Psi(|f|)$. When $0<|c|<1$, we have {$2^{-(k+1)2^{k+1}}\leq |c|<2^{-k2^{k}}$} for some $k \in \N\cup\{0\}$. Then by definition of class $\LL$ we have
    $$
     \Psi(|cf|) \lesssim \Psi(2^{-(k+1)2^{k+1}}|f|) \lesssim C^{k}\Psi(|f|) \lesssim C^{\log_2 \log_2 \frac{1}{|c|}} \Psi(|f|),
     $$
     giving us
     $$
\lVert cf \rVert^p_{L_{\mu,\Psi}^{p}}\lesssim |c|^p(1+C^{\log_2 \log_2 \frac{1}{|c|}})\lVert f \rVert^p_{L_{\mu,\Psi}^{p}}
     $$
     and the statement follows.
    \end{proof}

The following characterization for the compactness on the weighted Bergman space setting is well-known but we include a proof for the convenience of the readers.

\begin{proposition}\label{cptseq}
Let $0<p\le q<\infty$, $\omega \in \mathcal{D}$ and $\Phi,\Psi\in\LL$.
Then the canonical inclusion map $I: A^p_{\om,\Psi}\to L^q_{\mu,\Phi}$ is compact if and
only if $\lim_{n\to\infty} \|f_n\|_{L^q_{\mu,\Phi}}=0$ for each bounded sequence $\{f_n\}$ in $A^p_{\om,\Psi}$ that converges to 0
uniformly on compact subsets on $\D$.
\end{proposition}

\begin{proof}
Suppose $I$ is compact and $\{f_n\}$ is a bounded sequence in $A^p_{\om,\Psi}$ that converges to $0$ uniformly on compact subsets on $\D$.
Let $\{g_n\}$ be a subsequence of $\{f_n\}$. Since $I$ is compact, there exists a subsequence $\{I(g_{n_k})\}$ converging to $G$
in $L^q_{\mu,\Phi}$. Note that $G=0$ since $I(g_{nk})=g_{n_k}$ and $\{g_n\}$ converges to $0$ uniformly on compact subsets on $\D$.
Therefore every subsequence of $\{f_n\}$ has a subsequence converging to $0$ in $L^q_{\mu,\Phi}$. This implies that $\{f_n\}$ itself converges to $0$ in $L^q_{\mu,\Phi}$.

For the sufficiency, let $\|f_n\|_{A^p_{\om,\Psi}}\le M_0$ for some $M_0$.
Then $\{f_n\}$ is uniformly bounded on compact subsets of $\D$ by Lemma \ref{Growthf}.
Therefore Montel's theorem guarantees the existence of a subsequence $\{f_{n_k}\}$ which converges uniformly on compact subsets of $\mathbb D$ to some analytic function $f$. By Proposition~\ref{quasinorm}, for each $r\in (0,1)$, we have
\begin{align*}
\int_{D(0,r)} |f|^p\Psi(|f|)\om\, dA
\le C \int_{D(0,r)} |f-f_{n_k}|^p\Psi(|f-f_{n_k}|)\om\, dA
 +C \|f_{n_k}\|_{A^p_{\om,\Psi}}^p
\le C_1
\end{align*}
for some $C_1=C_1(p, \om, \Psi, M_0)$ taking $n_k$ sufficiently large.
Therefore $f\in A^p_{\om,\Psi}$ by Fatou's lemma.
Since the subsequence converges uniformly on compact subsets, by the assumption $\displaystyle \lim_{n\to\infty}\|I(f-f_{n_k})\|_{L^q_{\mu,\Phi}}=0$.
Hence $I: A^p_{\om,\Psi}\to L^q_{\mu,\Phi}$ is compact.
\end{proof}

The following lemma contains useful characterizations of weights in the class $\DD$. For a proof, see \cite[Lemma~2.1]{PelSum14}, for example.
\begin{letterlemma}\label{Dhat}
Let $\om$ be a radial weight. Then the following statements are equivalent:
		\begin{itemize}
		\item[(i)] $\om\in\DD$;
		\item[(ii)] There exist $C=C(\om)>0$ and $\beta_0=\beta_0(\om)>0$ such that
		$$
		\widehat{\om}(r)\leq C\left(\frac{1-r}{1-t}\right)^{\beta}\widehat{\om}(t),\quad 0\leq r\leq t<1,
		$$
for all $\beta\geq\beta_0$;
		\item[(iii)] There exists $\lambda=\lambda(\omega) \geq 0$ such that
		$$
		\int_{\mathbb{D}} \frac{\omega(z)}{|1-\bar{\zeta} z|^{\lambda+1}} d A(z) \asymp \frac{\widehat{\omega}(\zeta)}{(1-|\zeta|)^\lambda}, \quad \zeta \in \mathbb{D} .
		$$
		\end{itemize}
\end{letterlemma}

With the aid of Lemma~\ref{Dhat} we can now construct suitable testing functions for the proof of the main result.

\begin{lemma}\label{fa in Ap}
Let $0<p<\infty$, $\omega \in \widehat{\mathcal{D}}$ and $\Psi\in\LL$. Then there exists $\gamma_0=\gamma_0(\omega,\Psi)$ such that, for each $\gamma > \gamma_0$, the function
    $$
    f_a(z)=\left(\left(\frac{1-|a|}{1-\overline{a}z}\right)^\gamma \frac{1}{\omega(S(a))\Psi(\frac{1}{\omega(S(a))})}\right)^{\frac1{p}},\quad z\in\D,
    $$
belongs to $A_{\omega,\Psi}^{p}$ and $\lVert f_a \rVert^p_{A_{\omega,\Psi}^{p}} \asymp 1$ for all $a\in\D$.
\end{lemma}

\begin{proof}
For the lower bound, by Proposition~\ref{properties} we have
				\begin{align*}
        \lVert f_a \rVert^p_{A_{\omega,\Psi}^{p}}
				&\asymp\int_{\D}\left(\frac{1-|a|}{|1-\overline{a}z|}\right)^\gamma
				\frac{\omega(z)}{\omega(S(a))\Psi(\frac{1}{\omega(S(a))})}\Psi\left(\left(\frac{1-|a|}{|1-\overline{a}z|}\right)^\gamma \frac{1}{\omega(S(a))\Psi(\frac{1}{\omega(S(a))})}\right)\,dA(z)\\
        &\gtrsim\int_{S(a)} \frac{\omega(z)}{\omega(S(a))\Psi(\frac{1}{\omega(S(a))})}\Psi\left(\frac{1}{\omega(S(a))\Psi(\frac{1}{\omega(S(a))})}\right)\,dA(z)\\
        &=\frac{\Psi\left(\frac{1}{\omega(S(a))\Psi(\frac{1}{\omega(S(a))})}\right)}{\Psi(\frac{1}{\omega(S(a))})} \asymp 1,\quad a\in\D.
				\end{align*}

For the upper bound, assume first that $\Psi$ is essentially increasing. By Proposition~\ref{properties}(ii), (iii) and Lemma~\ref{Dhat}(iii), we have
				\begin{align*}
        \lVert f_a \rVert^p_{A_{\omega,\Psi}^{p}}
				&\asymp\int_{\D}\left(\frac{1-|a|}{|1-\overline{a}z|}\right)^\gamma
				\frac{\omega(z)}{\omega(S(a))\Psi(\frac{1}{\omega(S(a))})}\Psi\left(\left(\frac{1-|a|}{|1-\overline{a}z|}\right)^\gamma \frac{1}{\omega(S(a))\Psi(\frac{1}{\omega(S(a))})}\right)\,dA(z)\\
        &\lesssim \int_{\D}\left(\frac{1-|a|}{|1-\overline{a}z|}\right)^\gamma
				\frac{\omega(z)}{\omega(S(a))\Psi(\frac{1}{\omega(S(a))})}\Psi\left(\frac{1}{\omega(S(a))\Psi(\frac{1}{\omega(S(a))})}\right)\,dA(z)\\
        & \asymp \frac{\Psi\left(\frac{1}{\omega(S(a))\Psi(\frac{1}{\omega(S(a))})}\right)}{\Psi(\frac{1}{\omega(S(a))})}
        \asymp 1,\quad a\in\D.
				\end{align*}
Next, assume that $\Psi$ is essentially decreasing. Let $N=N(a)\in\N$ be the smallest integer such that $2^N(1-|a|)\ge1$. Define
				\begin{align*}
        S_0(a)=S(a), \quad
    S_k(a)&=S((1-2^k(1-|a|))e^{i \text{arg}a}),\quad k=1,\ldots, N-1,
    \end{align*}
and set $S_k(a)=\D$ if $k\ge N$. Then
    \begin{equation}\label{xx}
    |1-\overline{a}z| \asymp 2^k(1-|a|),\quad z \in S_k(a)\setminus S_{k-1}(a),\quad  k=1,\ldots, N.
    \end{equation}
Since $\Psi\in\LL$,  by Proposition~\ref{properties}(i) there exists a constant $C=C(\Psi)>0$ such that
    \begin{equation}\label{xxx}
   \frac{1}{C^k}\Psi(x) \le\Psi\left(\frac{x}{2^k}\right)\le C^k\Psi(x),\quad 0\leq x<\infty.
    \end{equation}
Therefore, by employing \eqref{xx}, \eqref{xxx}, Proposition \ref{properties}(ii)-(iii) and the hypothesis $\om\in\DD$, we deduce the existence of $C_1=C_1(\Psi)>0$ such that
    \begin{align*}
        \lVert f_a \rVert^p_{A_{\omega,\Psi}^{p}}&\asymp\int_{\D}\left(\frac{1-|a|}{|1-\overline{a}z|}\right)^\gamma \frac{\omega(z)}{\omega(S(a))\Psi(\frac{1}{\omega(S(a))})}\Psi\left(\left(\frac{1-|a|}{|1-\overline{a}z|}\right)^\gamma \frac{1}{\omega(S(a))\Psi(\frac{1}{\omega(S(a))})}\right)\,dA(z)\\
        & \le\sum_{k=1}^{N}\int_{S_k(a) \setminus S_{k-1}(a)}\left(\frac{1-|a|}{|1-\overline{a}z|}\right)^\gamma \frac{\omega(z)}{\omega(S(a))\Psi(\frac{1}{\omega(S(a))})}\\
				&\quad\cdot\Psi\left(\left(\frac{1-|a|}{|1-\overline{a}z|}\right)^\gamma \frac{1}{\omega(S(a))\Psi(\frac{1}{\omega(S(a))})}\right)\,dA(z) \\
        & \qquad + \int_{S(a)}\left(\frac{1-|a|}{|1-\overline{a}z|}\right)^\gamma \frac{\omega(z)}{\omega(S(a))\Psi(\frac{1}{\omega(S(a))})}\\
				&\quad\cdot\Psi\left(\left(\frac{1-|a|}{|1-\overline{a}z|}\right)^\gamma \frac{1}{\omega(S(a))\Psi(\frac{1}{\omega(S(a))})}\right)\,dA(z)\\
        & \asymp \sum_{k=1}^{N}\int_{S_k(a) \setminus S_{k-1}(a)}\frac{1}{2^{k\gamma}} \frac{\omega(z)}{\omega(S(a))\Psi(\frac{1}{\omega(S(a))})}\Psi\left(\frac{1}{2^{k\gamma}\omega(S(a))\Psi(\frac{1}{\omega(S(a))})}\right)\,dA(z)+1\\
                & \lesssim \sum_{k=0}^{N} \left(\frac{C_1}{2^{\gamma}}\right)^k \frac{\omega(S_k(a))}{\omega(S(a))},  \qquad a\in\D. \end{align*}
Proof is completed by taking $\gamma$ sufficiently large since $\omega(S_k(a))\le C^k\omega(S(a))$ by Lemma \ref{Dhat}(ii).
\end{proof}

We will also use the following characterization of the class $\Dd$, the proof of which can be found in \cite{PRS1}.

\begin{letterlemma}\label{Dcheck}
Let $\omega$ be a radial weight. Then $\om\in\Dd$ if and only if there exist $C=C(\omega)>0$ and $\beta=\beta(\omega)>0$ such that
	$$
	\widehat{\omega}(t) \leq C\left(\frac{1-t}{1-r}\right)^\beta \widehat{\omega}(r), \quad 0 \leq r \leq t<1.
	$$
\end{letterlemma}

In the proof of the main result we need to understand the behavior of the auxiliary weight $W=\Psi\left(\frac{1}{1-|\cdot|}\right)\om$. To do this, we will use the next result.

\begin{lemma}\label{LemmaPsiOm}
Let $\omega \in \mathcal{D}$ and $\Psi\in\LL$. Then
\begin{equation}\label{WeightEstimate}
\int_{r}^{1}\Psi\left(\frac{1}{1-s}\right)\frac{\widehat{\omega}(s)}{1-s}\,ds
\lesssim \int_{r}^{1}\Psi\left(\frac{1}{1-s}\right)\omega(s)\,ds,\quad 0\leq r<1.
\end{equation}
\end{lemma}

\begin{proof}
    Assume first that $\Psi$ is essentially decreasing. On one hand, since $\om\in\DDD\subset\Dd$, by Lemma \ref{Dcheck}, there exists a $\beta=\beta(\om)>0$ such that
\begin{align*}
    \int_{r}^{1}\Psi\left(\frac{1}{1-s}\right)\frac{\widehat{\omega}(s)}{1-s}\,ds &\lesssim \Psi\left(\frac{1}{1-r}\right)\int_r^1\frac{\whw(s)}{(1-s)^{1-\beta+\beta}}\,ds\\
    &\lesssim \Psi\left(\frac{1}{1-r}\right)\frac{\whw(r)}{(1-r)^{\beta}}\int_r^1\frac{1}{(1-s)^{1-\beta}}\,ds\lesssim \Psi\left(\frac{1}{1-r}\right)\whw(r).
\end{align*}
On the other hand, it follows from Lemma \ref{ThetaIncreasing} that $s\mapsto\frac{1}{(1-s)^{\beta}}\Psi\left(\frac{1}{1-s}\right)$ is essentially increasing. Then, by integrating by parts we obtain
		\begin{align*}
    \int_{r}^{1}\Psi\left(\frac{1}{1-s}\right)\omega(s)\,ds
		&=\int_{r}^{1}\Psi\left(\frac{1}{1-s}\right)\omega(s)\frac{(1-s)^{\beta}}{(1-s)^{\beta}}\,ds\\
		&\gtrsim \Psi\left(\frac{1}{1-r}\right) \frac{1}{(1-r)^{\beta}}\int_{r}^{1}(1-s)^{\beta}\omega(s)\,ds\\
    &=\Psi\left(\frac{1}{1-r}\right) \frac{1}{(1-r)^{\beta}} \left((1-r)^{\beta}\widehat{\omega}(r)
		+\int_{r}^{1} \beta (1-s)^{\beta-1}\whw(s)\,ds \right) \\
    &\ge\Psi\left(\frac{1}{1-r}\right)\widehat{\omega}(r),\quad 0\le r<1.
		\end{align*}
This finishes the proof of the case when $\Psi$ is essentially decreasing.

Assume next that $\Psi$ is essentially increasing. Then, by Lemma~\ref{Dcheck}, there exists an $\alpha=\alpha(\om)>0$ such that
\begin{align*}
    \int_{r}^{1}\Psi\left(\frac{1}{1-s}\right)\frac{\widehat{\omega}(s)}{1-s}\,ds &\lesssim
    \frac{\widehat{\omega}(r)}{(1-r)^{\alpha}}\int_{r}^{1}\Psi\left(\frac{1}{1-s}\right)\frac{1}{(1-s)^{1-\alpha}}\,ds,\quad 0\le r<1.
\end{align*}
Now, for $\beta<\alpha$, the function $s\mapsto(1-s)^{\beta}\Psi\left(\frac{1}{1-s}\right)$ is essentially decreasing by Lemma~\ref{ThetaIncreasing}. Thus
\begin{align*}
    \frac{\widehat{\omega}(r)}{(1-r)^{\alpha}}\int_{r}^{1}\Psi\left(\frac{1}{1-s}\right)\frac{(1-s)^{\beta}}{(1-s)^{1-\alpha+\beta}}\,ds \lesssim
    \widehat{\omega}(r)\Psi\left(\frac{1}{1-r}\right) \lesssim \int_{r}^{1}\Psi\left(\frac{1}{1-s}\right)\omega(s)\,ds
\end{align*}
for all $0\le r<1$. The claim follows.
\end{proof}
{It should be noted that it can easily be proven that $\Psi\left(\frac{1}{1-r}\right)\widehat{\omega}(r) \asymp \int_{r}^{1}\Psi\left(\frac{1}{1-s}\right)\omega(s)\,ds$ using similar methods as in the above proof. We will use this fact in the proof of Proposition \ref{setequiv}.}

\medskip

\Prf \emph{Theorem~\ref{mainresult}(i)}. Assume first that $\mu$ is a $(q,\Phi)$-Carleson measure for $A_{\omega,\Psi}^{p}$. Then Proposition~\ref{properties}(iii) and Lemma~\ref{fa in Ap} yield
				\begin{equation}\label{eqbdd}
				\begin{split}
        &\quad \frac{\mu(S(a))\Phi(\frac{1}{\omega(S(a))})}{\left(\omega(S(a))\Psi(\frac{1}{\omega(S(a))})\right)^{\frac{q}{p}}}\\
				&\asymp \int_{S(a)} \frac{1}{{\left(\omega(S(a))\Psi(\frac{1}{\omega(S(a))})\right)^{\frac{q}{p}}}}\Phi\left(\frac{1}{\omega(S(a))\Psi(\frac{1}{\omega(S(a))})}\right)\,d\mu(z)\\
        &\asymp\int_{S(a)}\left(\frac{1-|a|}{|1-\overline{a}z|}\right)^{\frac{q\gamma}{p}}\frac{1}{{\left(\omega(S(a))\Psi(\frac{1}{\omega(S(a))})\right)^{\frac{q}{p}}}}\Phi\left(\left(\frac{1-|a|}{|1-\overline{a}z|}\right)^{\gamma} \frac{1}{\omega(S(a))\Psi(\frac{1}{\omega(S(a))})}\right)\,d\mu(z)\\
        &\leq \int_{\D}\left(\frac{1-|a|}{|1-\overline{a}z|}\right)^\frac{q\gamma}{p} \frac{1}{{\left(\omega(S(a))\Psi(\frac{1}{\omega(S(a))})\right)^{\frac{q}{p}}}}\Phi\left(\left(\frac{1-|a|}{|1-\overline{a}z|}\right)^\gamma \frac{1}{\omega(S(a))\Psi(\frac{1}{\omega(S(a))})}\right)\,d\mu(z)\\
        &\asymp \|f_a\|^q_{L^q_{\mu, \Phi}}\lesssim1,\quad a\in\D.
        \end{split}
				\end{equation}

Conversely, assume \eqref{bdd}. Recall that the pseudohyperbolic disc is the set $\Delta(a,r)=\{z\in \D:|(a-z)/(1-\overline{a}z)|<r\}$. It is well known that $\Delta(a,r)$ is an Euclidean disc $D(A,R)$ centered at $A=(1-r^2)a/(1-r^2|a|^2)$ and of radius $R=(1-|a|^2)r/(1-r^2|a|^2)$. Since $\om\in\DDD$ by the hypothesis, we have $\om(S(a))\asymp\om(\Delta(a,r))$ for all $a\in\D$, provided $r=r(\om)\in(0,1)$ is sufficiently large. This is an easy consequence of the identity $\Delta(a,r)=D(A,R)$ and Lemmas~\ref{Dhat} and~\ref{Dcheck}. Further, we observe that once $r$ is fixed, there exists $K=K(r)>1$ such that $\Delta(a',r)\subset S(a)$ for $a'=\left(1-\frac{1-|a|}{K}\right)\frac{a}{|a|}$ for all $a\in\D$. Then it follows from  \eqref{bdd} and Proposition \ref{properties}(i) that
				\begin{equation}\label{bdd1}
				\sup_{a \in \D}\frac{\mu(\Delta(a,r))\Phi\left(\frac{1}{\omega(S(a))}\right)}{\left(\omega(\Delta(a,r))\Psi\left(\frac{1}{\omega(S(a))}\right)\right)^{\frac{q}{p}}} < \infty.
				\end{equation}
This is the condition that we will use to show that $\mu$ is a $(q,\Phi)$-Carleson measure for $A_{\omega,\Psi}^{p}$. Recall that for $r=r(\om)\in(0,1)$ appearing in \eqref{bdd1} we have $\om(S(a))\asymp\om(\Delta(a,r))$ for all $a\in\D$. This will be repeatedly used during the rest of the proof.

Let $f \in A_{\omega,\Psi}^{p}$ with $\|f\|_{A_{\omega,\Psi}^{p}}=M>0$.  The proof will be divided into three separate cases. To shorten the notation, we will write $\om_{[x]}(z)=\om(z)(1-|z|)^x$ for all $z\in\D$ and $x\in\mathbb{R}$.

\medskip
{\textbf{Case I: $\Phi$ and $\Psi$ are essentially increasing.}}
By the subharmonicity of $|f|^p$, Minkowski's inequality, Lemma~\ref{Growthf}, Proposition~\ref{properties}(ii) and \eqref{bdd1}, we deduce
        \begin{equation}\label{eq1}
				\begin{split}
        \lVert f \rVert_{L_{\mu, \Phi}^{q}}^{q}&=\int_{\D}|f(z)|^q \Phi(|f(z)|) \,d\mu(z)\\
        &\lesssim \int_{\D} \left(\int_{\Delta(z,r)}\frac{|f(\zeta)|^p}{(1-|\zeta|)^2} \,dA(\zeta) \left(\Phi(|f(z)|)\right)^{\frac{p}{q}} \right)^{\frac{q}{p}}\,d\mu(z)\\
        &\leq \left(\int_{\D} \frac{|f(\zeta)|^p}{(1-|\zeta|)^2}\left(\int_{\Delta(\zeta,r)} \Phi(|f(z)|)\,d\mu(z)\right)^{\frac{p}{q}}\,dA(\zeta)\right)^{\frac{q}{p}}\\
        &\lesssim\left(\int_{\D} \frac{|f(\zeta)|^p}{(1-|\zeta|)^2}\left(\int_{\Delta(\zeta,r)} \Phi\left(\frac{M^p+1}{\om(S(z))}\right)\,d\mu(z)\right)^{\frac{p}{q}}\,dA(\zeta)\right)^{\frac{q}{p}}\\
        &\asymp\left(\int_{\D} \frac{|f(\zeta)|^p}{(1-|\zeta|)^2} \left( \mu(\Delta(\zeta,r))\Phi\left(\frac{1}{\omega(S(\zeta))}\right)\right)^{\frac{p}{q}}\,dA(\zeta)\right)^{\frac{q}{p}}\\
        &\lesssim \left(\int_{\D}|f(\zeta)|^p\Psi\left(\frac{1}{\omega(S(\zeta))}\right)\frac{\omega(\Delta(\zeta,r))}{(1-|\zeta|)^2}  \,dA(\zeta)\right)^{\frac{q}{p}}\\
        &\asymp\left(\int_{\D}|f(\zeta)|^p\Psi\left(\frac{1}{\omega(S(\zeta))}\right)\frac{\omega(S(\z))}{(1-|\zeta|)^2}  \,dA(\zeta)\right)^{\frac{q}{p}}.
        \end{split}
        \end{equation}
Since the quantity $\Psi\left(\frac{1}{\omega(S(\z))}\right) \frac{\omega\left(S(\z)\right)}{(1-|\z|)^2}$ depends upon $|\z|$ only, an integration by parts together with Lemma~\ref{LemmaPsiOm} yields
	\begin{equation}\label{eq2}
	\int_{\D}|f(\zeta)|^p\Psi\left(\frac{1}{\omega(S(\zeta))}\right)\frac{\omega\left(S(\zeta)\right)}{(1-|\zeta|)^2}\,dA(\zeta)
	\lesssim \int_{\D}|f(\zeta)|^p\Psi\left(\frac{1}{\omega(S(\zeta))}\right)\omega(\zeta)\,dA(\zeta),
	\end{equation}
see \cite[Lemma~8]{PRV} for a proof in a general case. Thus
	\begin{equation}\label{eq1'}
  \begin{split}
  \lVert f \rVert_{L_{\mu, \Phi}^{q}}^{p}\lesssim \int_{\D}|f(\zeta)|^p\Psi\left(\frac{1}{\omega(S(\zeta))}\right)\omega(\zeta)\,dA(\zeta).
  \end{split}
  \end{equation}
Define $E_{\veps}=E_{\veps}(f)=\{z\in \D:|f(z)|<\frac{1}{(1-|z|)^{\veps}}\}$, where $\veps=\veps(p,\omega)>0$ will be fixed later. By Lemma~\ref{Dhat}(ii), there exists an {$\alpha=\alpha(\om)>0$} such that $(1-|a|)^{1+\alpha}\lesssim\omega(S(a)) \lesssim(1-|a|)$ for all $a\in\D$. Bearing in mind that $\Psi$ is essentially increasing by the assumption, and using \eqref{eq1'}, Proposition~\ref{properties}(ii) and Lemma~\ref{GrowthPsi}, we deduce
		\begin{align*}
    \lVert f \rVert_{L_{\mu, \Phi}^{q}}^{p}
		&\lesssim\int_{\D}|f(\zeta)|^p\Psi\left(\frac{1}{\omega(S(\zeta))}\right)\omega(\zeta)\,dA(\zeta)
		\asymp\int_{\D}|f(\zeta)|^p\Psi\left(\frac{1}{1-|\zeta|}\right)\omega(\zeta)\,dA(\zeta)\\
    &=\int_{E_{\veps}}|f(\zeta)|^p\Psi\left(\frac{1}{1-|\zeta|}\right)\omega(\zeta)\,dA(\zeta)
		+\int_{\D\setminus E_{\veps}}|f(\zeta)|^p\Psi\left(\frac{1}{1-|\zeta|}\right)\omega(\zeta)\,dA(\zeta)\\
    &\le\int_{E_{\veps}}\frac{1}{(1-|\zeta|)^{p\veps}}\Psi\left(\frac{1}{1-|\zeta|}\right)\omega(\zeta)\,dA(\zeta)
		+\int_{\D\setminus E_{\veps}}|f(\zeta)|^p\Psi\left(|f(\zeta)|\right)\omega(\zeta)\,dA(\zeta)\\
    &\lesssim\int_{\D}\Psi\left(\frac{1}{1-|\zeta|}\right)\omega_{[-p\veps]}(\zeta)\,dA(\zeta)
		+\lVert f\rVert^p_{A_{\omega,\Psi}^{p}}\lesssim\int_{\D}\omega_{[-(p\veps+\eta)]}(\zeta)\,dA(\zeta)
		+\lVert f\rVert^p_{A_{\omega,\Psi}^{p}}
		\end{align*}
for each fixed $\eta>0$. It is well known that for each $\om\in\DDD$ there exists a $\delta_0=\delta_0(\om)>0$ such that $\om_{[-\d]}\in\DDD$ for all $0<\delta<\delta_0$. This is an easy consequence of an integration by parts and Lemmas~\ref{Dhat} and \ref{Dcheck}, see \cite[Lemma~2]{PR2023} for a more general result. By choosing $\veps>0$ and $\eta>0$ small enough such that $p\veps+\eta<\delta_0$, we have $\omega_{[-(p\veps+\eta)]}\in\DDD$, and thus $\omega_{[-(p\veps+\eta)]}\in L^1$. It follows that $\|f\|^p_{L^q_{\mu,\Phi}}\lesssim1$, where the suppressed constant depends on $M,p,q,\om,\Phi$ and $\Psi$.

\medskip

{\textbf{Case II: $\Phi$ is essentially increasing, and $\Psi$ is essentially decreasing.}} Let $E_{\veps,p}=E_{\veps,p}(f)=\{z\in\D:|f(z)|<\frac{1}{(1-|z|)^{\veps/p}}\}$ with $\veps=\veps(\omega)<\delta_0$, where $\delta_0$ is that of Case I. Then \eqref{eq1'}, Proposition~\ref{properties}(ii) and Lemma~\ref{Growthf} yield
		\begin{align*}
		\lVert f \rVert_{L_{\mu, \Phi}^{q}}^p
		&\lesssim\int_{E_{\veps,p}}|f(\zeta)|^p\Psi\left(\frac{1}{1-|\zeta|}\right)\omega(\zeta)\,dA(\zeta)
		+\int_{\D\setminus E_{\veps,p}}|f(\zeta)|^p\Psi\left(\frac{1}{1-|\zeta|}\right)\omega(\zeta)\,dA(\zeta)\\
&\asymp\int_{E_{\veps,p}}|f(\zeta)|^p\Psi\left(\frac{1}{1-|\zeta|}\right)\omega(\zeta)\,dA(\zeta)
		+\int_{\D\setminus E_{\veps,p}}|f(\zeta)|^p\Psi\left(\frac{M^p+1}{\om(S(\z))}\right)\omega(\zeta)\,dA(\zeta)\\ &\lesssim\int_{E_{\veps,p}}\omega_{[-\veps]}(\zeta)\,dA(\zeta)
		+\int_{\D\setminus E_{\veps,p}}|f(\zeta)|^p\Psi\left(|f(\zeta)|\right)\omega(\zeta)\,dA(\zeta)
		\lesssim 1+\|f\|_{A^p_{\om,\Psi}}^p,
		\end{align*}
 where the suppressed constant depends on $M,p,q,\om,\Phi$ and $\Psi$ and thus this case is proved.

\medskip

{\textbf{Case III: $\Phi$ is essentially decreasing.}} Let $E_{\veps}=E_{\veps}(f)=\{z\in \D: |f(z)|<\frac{1}{(1-|z|)^{\veps}}\}$, where $\veps=\veps(q,\omega)>0$ will be fixed later, and write
		\begin{align*}
		\lVert f \rVert_{L_{\mu, \Phi}^{q}}^q
		&=\int_{E_\veps}|f(z)|^q \Phi(|f(z)|) \,d\mu(z)
    +\int_{\D \setminus E_{\veps}}|f(z)|^q \Phi(|f(z)|)\,d\mu(z).
		\end{align*}
For the first integral, we observe that Fubini's theorem and \eqref{bdd1} yield
		\begin{align*}
    \int_{E_\veps}|f(z)|^q \Phi(|f(z)|) \,d\mu(z)
    &\lesssim\int_{E_\veps}\frac{d\mu(z)}{(1-|z|)^{q\veps}}
		\lesssim\int_{\D}\left(\int_{\Delta(z,r)}\frac{dA(\zeta)}{(1-|\zeta|)^{q\veps+2}}\right)d\mu(z)\\
    &=\int_{\D}\frac{\mu(\Delta(\zeta,r))}{(1-|\zeta|)^{q\veps+2}}\,dA(\zeta)\\
		&\lesssim\int_{\D}\frac{\omega(\Delta(\zeta,r))^\frac{q}{p}\Psi\left(\frac{1}{\omega(S(\zeta))}\right)^\frac{q}{p}}
		{\Phi\left(\frac{1}{\omega(S(\zeta))}\right)}\frac{dA(\zeta)}{(1-|\zeta|)^{q\veps+2}},
		\end{align*}
where
	$$
	\frac{\Psi\left(\frac{1}{\omega(S(\zeta))}\right)^\frac{q}{p}}{\Phi\left(\frac{1}{\omega(S(\zeta))}\right)}\lesssim\frac{1}{(1-|\zeta|)^\eta},\quad \zeta\in\D,
	$$
for each fixed $\eta>0$ by Lemma~\ref{GrowthPsi}. By using $\om(\Delta(\zeta,r))\asymp\om(S(\zeta))$ for all $a\in\D$, and Lemma~\ref{Dcheck}, we deduce
	\begin{equation*}
	\begin{split}
	\int_{E_\veps}|f(z)|^q \Phi(|f(z)|) \,d\mu(z)
	&\lesssim\int_\D\frac{\widehat{\om}(\zeta)^\frac{q}{p}(1-|\zeta|)^\frac{q}{p}}{(1-|\zeta|)^{\eta+q\veps+2}}\,dA(\zeta)
	\lesssim\int_0^1\frac{\widehat{\om}(r)}{(1-r)^{\eta+q\veps+1}}\,dr\\
	&\lesssim\int_0^1\frac{dr}{(1-r)^{\eta+q\veps+1-\beta}}
	\end{split}
	\end{equation*}
for some $\beta=\beta(\om)>0$. By choosing $\eta>0$ and $\veps>0$ small enough, the last integral converges. To estimate the second integral, we use the subharmonicity of $|f|^p$, Minkowski's inequality and \eqref{eq2} to obtain
				\begin{align*}
        \int_{\D \setminus E_{\veps}}|f(z)|^q \Phi(|f(z)|) \,d\mu(z)
        &\lesssim\int_{\D \setminus E_{\veps}}|f(z)|^q \Phi\left(\frac{1}{\om(S(z))}\right) \,d\mu(z)\\
        &\lesssim \int_{\D \setminus E_{\veps}} \left(\int_{\Delta(z,r)}\frac{|f(\zeta)|^p}{(1-|\zeta|)^2}\,dA(\zeta) \Phi\left(\frac{1}{\om(S(z))}\right)^{\frac{p}{q}} \right)^{\frac{q}{p}}\,d\mu(z)\\
        &\le\left(\int_{\D} \frac{|f(\zeta)|^p}{(1-|\zeta|)^2}\left(\int_{\Delta(\zeta,r)} \Phi\left(\frac{1}{\om(S(z))}\right)\,d\mu(z)\right)^{\frac{p}{q}}\,dA(\zeta)\right)^{\frac{q}{p}}\\
        &\lesssim\left(\int_{\D}\frac{|f(\zeta)|^p}{(1-|\zeta|)^2}\left(\mu(\Delta(\zeta,r))\Phi\left(\frac{1}{\omega(S(\zeta))}\right)\right)^{\frac{p}{q}}\,dA(\zeta)\right)^{\frac{q}{p}}\\
        &\lesssim\left(\int_{\D}|f(\zeta)|^p\Psi\left(\frac{1}{\omega(S(\zeta))}\right)\frac{\omega(\Delta(\zeta,r))}{(1-|\zeta|)^2}\,dA(\zeta)\right)^{\frac{q}{p}}\\
				&\lesssim\left(\int_{\D}|f(\zeta)|^p\Psi\left(\frac{1}{\omega(S(\zeta))}\right)\omega(\zeta)\,dA(\zeta)\right)^{\frac{q}{p}}.
				\end{align*}
By combining the estimates above we have
		$$
		\lVert f \rVert_{L_{\mu, \Phi}^{q}}^{p}
		\lesssim\left(\int_{\D}|f(\zeta)|^p\Psi\left(\frac{1}{\omega(S(\zeta))}\right)\omega(\zeta)\,dA(\zeta)\right)^{\frac{q}{p}}+1.
		$$
If $\Psi$ is essentially increasing (resp. decreasing), we may finish the proof by arguing as in Case~I (resp. Case~II). The proof of Theorem~\ref{mainresult}(i) is now complete.\hfill$\Box$

\medskip

\Prf \emph{Theorem~\ref{mainresult}(ii)}. Assume first that $I: A^p_{\om,\Psi}\to L^q_{\mu,\Phi}$ is compact. The functions $f_a$ defined in Lemma~\ref{fa in Ap} satisfy $\|f_a\|_{A^p_{\om,\Psi}}\asymp1$ for all $a\in\D$, and certainly $f_a\to 0$ on compact subsets of $\D$, as $|a|\to 1^-$. Therefore $\|f_a\|_{L^q_{\mu,\Phi}}\to0$ as $|a|\to 1^-$, and hence \eqref{eqbdd} yields
	$$
	0=\lim_{|a|\to 1^-}\|f_a\|^q_{L^q_{\mu,\Phi}}
	\gtrsim \lim_{|a|\to 1^-}\frac{\mu(S(a))\Phi\left(\frac{1}{\omega(S(a))}\right)}{\left(\omega(S(a))\Psi\left(\frac{1}{\omega(S(a))}\right)\right)^{\frac{q}{p}}}.
	$$
Proposition~\ref{cptseq} now completes the proof of the necessity.

For the sufficiency, assume \eqref{cpt}. Then a reasoning similar to that employed at the beginning of the proof of Theorem~\ref{mainresult}(i) shows that
				\begin{equation}\label{cpt1}
				\lim_{|a|\to1^-}\frac{\mu(\Delta(a,r))\Phi\left(\frac{1}{\omega(S(a))}\right)}{\left(\omega(\Delta(a,r))\Psi\left(\frac{1}{\omega(S(a))}\right)\right)^{\frac{q}{p}}}=0.
				\end{equation}
Let $\{f_n\}$ be a norm bounded sequence in $A^p_{\om,\Psi}$ such that $f_n\to0$ uniformly on compact subsets of $\D$, as $n\to\infty$. To prove the assertion, by Proposition~\ref{cptseq}, it suffices to show that
	\begin{equation}\label{eqcpt}
	\lim_{n\to\infty}\|f_n\|_{L^q_{\mu,\Phi}}=0.
	\end{equation}
Let $R\in(0,1)$ to be fixed later and $D(0,R)=\{z\in\D: |z|\leq R\}$. Then
\begin{align*}
\|f_n\|_{L^q_{\mu,\Phi}}^q&=\int_\D |f_n(z)|^q\Phi(|f_n(z)|)\,d\mu(z)\\&=\left(\int_{D(0,R)+\D\setminus D(0,R)}\right)|f_n(z)|^q\Phi(|f_n(z)|)\,d\mu(z)=I_1+I_2.
\end{align*}
Since $\mu$ is a finite measure, by Lemmas~\ref{GrowthPsi} and~\ref{Growthf}, and the dominated convergence theorem, for arbitrary $\epsilon>0$ and $R\in(0,1)$, there exists an $N=N(\epsilon)\in\N$ such that $I_1\lesssim\epsilon$ for all $n\ge N$. To estimate $I_2$, we first observe that, by \eqref{cpt1}, there exists an $R=R(\epsilon)\in(0,1)$ such that
	\begin{equation}\label{eq3}
	\frac{\mu(\Delta(a,r))\Phi\left(\frac{1}{\omega(S(a))}\right)}
	{\left(\omega(\Delta(a,r))\Psi\left(\frac{1}{\omega(S(a))}\right)\right)^{\frac{q}{p}}}<\epsilon,\quad R<|a|<1.
	\end{equation}
Suppose first that $\Phi$ is essentially increasing. Imitating the reasoning in \eqref{eq1} and employing \eqref{eq3}, we deduce
	\begin{equation*}
  \begin{split}
  I_2&=\int_{\D\setminus D(0,R)}|f_n(z)|^q\Phi(|f_n(z)|)\,d\mu(z)\\
  &\lesssim\left(\int_{\D}\frac{|f_n(\z)|^p}{(1-|\z|)^2}
	\left(\int_{\Delta(\z,r)\setminus D(0,R)} \Phi(|f_n(z)|)\,d\mu(z)\right)^{\frac{p}{q}}\,dA(\z)\right)^{\frac{q}{p}}\\
  &\lesssim\epsilon\left(\int_{\D}|f_n(\z)|^p\Psi\left(\frac{1}{\omega(S(\z))}\right)\frac{\omega(\Delta(\z,r))}{(1-|\z|)^2}\,dA(\z)\right)^{\frac{q}{p}}
	\lesssim\epsilon.
  \end{split}
\end{equation*}
Here the last inequality follows from the same steps appearing in Cases~I and~II of the proof of Part~(i), and the hypothesis that $\{f_n\}$ is bounded in $A^p_{\om,\Psi}$. Therefore we eventually deduce $\eqref{eqcpt}$ when $\Phi$ is essentially increasing.

Suppose now $\Phi$ is essentially decreasing. Let $E_\veps=E_\veps(f_n)$ be defined as before. Similar arguments as were used to prove Case~III of Part~(i) yield
\begin{equation*}
    \begin{split}
    I_2&=\int_{\D\setminus D(0,R)}|f_n(z)|^q\Phi(|f_n(z)|)\,d\mu(z)\\
    &=\left(\int_{E_{\veps}\setminus D(0,R)}+\int_{(\D\setminus D(0,R))\setminus E_\veps}\right)|f_n(z)|^q\Phi(|f_n(z)|)\,d\mu(z)
    =I_{2,1}+I_{2,2}.
    \end{split}
\end{equation*}
To estimate $I_{2,1}$, we use arguments appearing in Case~III and \eqref{eq3} to find an $R=R(\epsilon)\in(0,1)$ such that
\begin{equation}\label{1}
    \begin{split}
        I_{2,1}&\lesssim \left(\int_{\D} \frac{1}{(1-|\z|)^{p\veps +2}} (\mu(\Delta(\z,r)\setminus D(0,R)))^{\frac{p}{q}}\,dA(\z)\right)^\frac{q}{p}\\
    &\lesssim \epsilon\left(\int_{\D}\frac{dA(\z)}{(1-|\z|)^{p\veps+\eta+1-\beta}}\right)^\frac{q}{p}\lesssim\epsilon,
    \end{split}
\end{equation}
where {similar steps were taken as in proof of Case III}. Likewise, for $I_{2,2}$, we get
		\begin{equation}\label{2}
    \begin{split}
    I_{2,2}
		&\lesssim\int_{(\D\setminus D(0,R))\setminus E_{\veps}}\left(\int_{\Delta(z,r)}\frac{|f_n(\z)|^p}{(1-|\z|)^2} \,dA(\z)
		\Phi(|f_n(z)|)^{\frac{p}{q}} \right)^{\frac{q}{p}}\,d\mu(z)\\
    &\lesssim\left(\int_{\D}\frac{|f_n(\z)|^p}{(1-|\z|)^2}\left(\int_{\Delta(\z,r)\setminus D(0,R)}
		\Phi\left(\frac{1}{\om(S(z)}\right)\,d\mu(z)\right)^{\frac{p}{q}}\,dA(\z)\right)^{\frac{q}{p}}
		\lesssim \epsilon.
    \end{split}
\end{equation}
Hence, by combining the estimate for $I_1$, \eqref{1} and \eqref{2}, we establish \eqref{eqcpt}. \hfill$\Box$\\

\begin{proposition}\label{setequiv}
Let $0<p<\infty, \Psi \in \LL$ and $\om \in \DDD$. Then there exists a weight $W=W(\om,\Psi) \in \DDD$ such that $A_{\om,\Psi}^p=A_W^p$ when considered as sets. Also the identity mappings $I: A_{\om,\Psi}^p \to A_W^p$ and $I: A_W^p \to A_{\om,\Psi}^p$ are bounded. Moreover, $\lim_{n\to\infty}\|f_n\|_{A^p_W}= 0$ if and only if $\lim_{n\to\infty} \|f_n\|_{A^p_{\om,\Psi}}=0$.
\end{proposition}

\begin{proof}
Define $W(z)=\Psi\left(\frac{1}{1-|z|}\right)\om(z)$ for all $z \in \D$. We first show that $W\in\DDD$. By the hypothesis $\om\in\DD$, Lemma~\ref{properties}(i) and the afternote of Lemma~\ref{LemmaPsiOm}, we have
    $$
    \widehat{W}(r)
		\asymp\Psi\left(\frac{1}{1-r}\right)\widehat{\om}(r)
		\lesssim\Psi\left(\frac{1}{1-\frac{1+r}{2}}\right)\widehat{\om}\left(\frac{1+r}{2}\right)
		\asymp\widehat{W}\left(\frac{1+r}{2}\right),\quad 0\leq r<1.
    $$
Therefore $W\in\DD$ by the definition. Let $\veps=\veps(\om)\in(0,\beta)$, where $\beta=\beta(\om)$ is that of Lemma~\ref{Dcheck}. Then Lemmas~\ref{Dcheck} and~\ref{ThetaIncreasing} imply
    \begin{equation*}
    \begin{split}
        \widehat{W}(t) &\asymp \Psi\left(\frac{1}{1-t}\right)\widehat{\om}(t)
				\lesssim \Psi\left(\frac{1}{1-t}\right)\left(\frac{1-t}{1-r}\right)^{\beta}\whw(r)\\
        &=\frac{\Psi\left(\frac{1}{1-t}\right)(1-t)^{\veps}}{\Psi\left(\frac{1}{1-r}\right)(1-r)^{\veps}}\left(\frac{1-t}{1-r}\right)^{\beta-\veps}\Psi\left(\frac{1}{1-r}\right)\whw(r) \\
        & \lesssim \left(\frac{1-t}{1-r}\right)^{\beta-\veps}\widehat{W}(r), \quad 0 \leq r < t < 1,
    \end{split}
    \end{equation*}
which shows $W \in \Dd$ by Lemma \ref{Dcheck}. Thus $W\in\DDD$.

By Proposition~\ref{properties}(i) and the afternote of Lemma~\ref{LemmaPsiOm} we have
    \begin{align*}
        \frac{W(S(a))}{\Psi\left(\frac{1}{1-|a|}\right) \om(S(a))}
				\asymp 1,\quad a\in\D.
    \end{align*}
Hence Theorem~\ref{mainresult}(i) implies that the identity mappings between $A_{\om,\Psi}^p$ and $A_W^p$ are bounded, and thus $A_{\om,\Psi}^p=A_W^p$ as sets.

For the last statement, assume first that $\Psi$ is essentially increasing. Let first $\{f_n\}$ be a bounded sequence in $A^p_W$. By the first part of the proof we know that $\{f_n\}$ is also a bounded sequence in $A^p_{\om,\Psi}$. By Lemma~\ref{Growthf} we have
	\begin{equation*}
	\|f_n\|^p_{A^p_{\om,\Psi}}
  \lesssim \int_{\D}|f_n(z)|\Psi\left(\frac{1}{1-|z|}\right)\om(z)\,dA(z)=\|f_n\|^p_{A^p_W},
	\end{equation*}
which shows that if $\|f_n\|^p_{A^p_W} \to 0$, then $\|f_n\|^p_{A^p_{\om,\Psi}} \to 0$.

Let now $\{f_n\}$ be a bounded sequence in $A^p_{\om,\Psi}$ for which $\|f_n\|^p_{A^p_{\om,\Psi}} \to 0$. Let $E^n_{\veps}=E^n_{\veps}(f)=\{z \in \D: |f_n(z)|<\frac{1}{(1-|z|)^{\veps/p}}\}$, where $\veps>0$ is fixed such that
		\begin{equation*}
    \int_{\D}\frac{1}{(1-|z|)^{\veps}}\Psi\left(\frac{1}{1-|z|}\right)\om(z)\,dA(z)<\infty.
		\end{equation*}
Then
		\begin{equation*}
    \begin{split}
    \|f_n\|^p_{A^p_W}
		&=\int_{\D}|f_n(z)|\Psi\left(\frac{1}{1-|z|}\right)\om(z)\,dA(z) \\
    &=\int_{E^n_{\veps}}|f_n(z)|\Psi\left(\frac{1}{1-|z|}\right)\om(z)\,dA(z)
		+\int_{\D \setminus E^n_{\veps}}|f_n(z)|\Psi\left(\frac{1}{1-|z|}\right)\om(z)\,dA(z)\\
    &\lesssim \int_{E^n_{\veps}}|f_n(z)|\Psi\left(\frac{1}{1-|z|}\right)\om(z)\,dA(z)+\|f_n\|^p_{A^p_{\om,\Psi}}.
    \end{split}
		\end{equation*}
By the proof of Lemma~\ref{Growthf}, we know that there is no `$+1$' in the numerator on the right hand side of the statement of the said lemma when $\Psi$ is essential increasing. Hence $|f_n|\to0$ pointwise in $\D$ when $\|f_n\|^p_{A^p_{\om,\Psi}}\to0$. Thus the dominated convergence theorem yields
\begin{equation}
    \int_{E^n_{\veps}}|f_n(z)|\Psi\left(\frac{1}{1-|z|}\right)\om(z)\,dA(z) \to 0,
\end{equation}
when $n \to \infty$. Therefore $\|f_n\|^p_{A^p_W} \to 0$ if $\|f_n\|^p_{A^p_{\om,\Psi}} \to 0$.

Assume next that $\Psi$ is essentially decreasing. Then
	\begin{equation*}
  \|f_n\|^p_{A^p_W}
	\lesssim\int_{\D}|f_n(z)|\Psi\left(|f_n(z)|\right)\om(z)\,dA(z)
	=\|f_n\|^p_{A^p_{\om,\Psi}},
	\end{equation*}
which shows that if $\|f_n\|^p_{A^p_{\om,\Psi}} \to 0$, then $\|f_n\|^p_{A^p_W} \to 0$.
Moreover, for the set $E^n_{\veps}=E^n_{\veps}(f)$, defined above, we have
	\begin{equation*}
  \begin{split}
  \|f_n\|^p_{A^p_{\om,\Psi}}
	&=\int_{\D}|f_n(z)|^p\Psi\left(|f_n(z)|\right)\om(z)\,dA(z)\\
  &\lesssim \int_{E^n_{\veps}}|f_n(z)|^p\Psi\left(|f_n(z)|\right)\om(z)\,dA(z)
	+\int_{\D \setminus E^n_{\veps}}|f_n(z)|\Psi\left(\frac{1}{1-|z|}\right)\om(z)\,dA(z)\\
  &\le\int_{E^n_{\veps}}|f_n(z)|^p\Psi\left(|f_n(z)|\right)\om(z)\,dA(z)+\|f_n\|^p_{A^p_W}.
  \end{split}
	\end{equation*}
Since $|f_n|\to 0$ pointwise in $\D$ when $\|f_n\|^p_{A^p_W}\to0$ by Lemma~\ref{Growthf} and its proof, we deduce
	\begin{equation*}
  \int_{E^n_{\veps}}|f_n(z)|^p\Psi\left(|f_n(z)|\right)\om(z)\,dA(z) \to 0,
	\end{equation*}
as $n \to \infty$. Therefore we have $\|f_n\|^p_{A^p_W} \to 0$ when $\|f_n\|^p_{A^p_{\om,\Psi}} \to 0$, which concludes the proof.
\end{proof}

Recall that in Proposition~\ref{Growthf} we provided a rough growth estimate for functions in $A^p_{\om,\Psi}$ when $\om\in\DD$. If $\om\in\DDD$, then we can give a more precise growth estimate, which is needed to prove Theorem~\ref{Tgbounded}. A sharp bound for the constant $C$ involved in the statement can certainly be tracked down in terms of $\|f\|_{A^p_{\om,\Psi}}$ but it is not relevant for our purposes, and therefore we leave it for interested readers.

\begin{lemma}\label{Growthf1}
Let $0<p<\infty$, $\omega \in \DDD$ and $\Psi\in\LL$. Then, for each $f\in A^p_{\om,\Psi}$, there exists a constant $C=C(p,\om,\Psi, \|f\|_{A^p_{\om,\Psi}})$ such that
\begin{equation}\label{growth1}
    |f(z)|^{p}\le \frac{C}{\om(S(z))
    \Psi\left(\frac{1}{\om(S(z))}\right)
    }, \quad f\in A_{\omega,\Psi}^{p},\quad z\in\D.
    \end{equation}
\end{lemma}

\begin{proof}
By the subharmonicity of $|f|^p$, and integration by parts with Lemma \ref{LemmaPsiOm}, we deduce
\begin{equation}\label{e1}
    \begin{split}
        |f(z)|^p &\lesssim \frac{1}{(1-|z|)^2}\int_{\Delta(z,r)}|f(\z)|^p\frac{\Psi\left(\frac{1}{1-|\z|}\right)}{\Psi\left(\frac{1}{1-|\z|}\right)}\frac{\whw(\z)}{\whw(\z)}\,dA(\z) \\&\asymp \frac{1}{\om(S(z))\Psi\left(\frac{1}{1-|z|}\right)}\int_{\Delta(z,r)}|f(\z)|^p\Psi\left(\frac{1}{1-|\z|}\right)\frac{\whw(\z)}{1-|\z|}\,dA(\z)\\
        &\lesssim \frac{1}{\om(S(z))\Psi\left(\frac{1}{1-|z|}\right)}\int_{\D}|f(\z)|^p\Psi\left(\frac{1}{1-|\z|}\right)\om(\z)\,dA(\z)\quad z\in\D.
    \end{split}
\end{equation}
Since $\om\in\DDD$ by the hypothesis, we are able to find a sufficiently small $\veps=\veps(\om)$, such that
\begin{equation}\label{e2}
 \int_{\D}\frac{1}{(1-|z|)^{\veps}}\Psi\left(\frac{1}{1-|\z|}\right)\om(z)\,dA(z)<\infty.
\end{equation}
When $\Psi$ is essentially increasing, let $E_{\veps,p}=\{z\in\D:|f(z)|<\frac{1}{(1-|z|)^{\veps/p}}\}$. Then \eqref{e1} and \eqref{e2} yield
\begin{equation}\label{zz}
    \begin{split}
        |f(z)|^p &\lesssim \frac{1}{\om(S(z))\Psi\left(\frac{1}{1-|z|}\right)}\left(\int_{E_{\veps,p}}+\int_{\D \setminus E_{\veps,p}}\right)|f(\z)|^p\Psi\left(\frac{1}{1-|\z|}\right)\om(\z)\,dA(\z)\\
        &\lesssim \frac{1}{\om(S(z))\Psi\left(\frac{1}{1-|z|}\right)}\left(1+\int_{\D \setminus E_{\veps,p}}|f(\z)|^p\Psi\left(|f(\z)|\right)\om(\z)\,dA(\z)\right)\\
          &\leq \frac{\lVert f \rVert^p_{A_{\omega,\Psi}^{p}}+1}{\om(S(z))\Psi\left(\frac{1}{1-|z|}\right)},\quad z\in\D.
    \end{split}
\end{equation}

When $\Psi$ is essentially decreasing, then \eqref{growth}, \eqref{e1} and Proposition~\ref{properties}(i),(ii) yield
    \begin{equation}\label{zzz}
    \begin{split}
           |f(z)|^p&\lesssim\frac{1}{\om(S(z))\Psi\left(\frac{1}{1-|z|}\right)}\int_{\D}|f(\z)|^p\Psi\left(\frac{1}{1-|\z|}\right)\om(\z)\,dA(\z)\\
           &\lesssim\frac{1}{\om(S(z))\Psi\left(\frac{1}{1-|z|}\right)}\int_{\D}|f(\z)|^p\Psi\left(\frac{|f(\z)|^{p-\delta}}{\|f\|^p_{A^p_{\om,\Psi}}+1}\right)\om(\z)\,dA(\z)\\
           &\lesssim \frac{\|f\|^p_{A^p_{\om,\Psi}}}{\om(S(z))\Psi\left(\frac{1}{1-|z|}\right)},\quad z\in\D,
       \end{split}
    \end{equation}
    where the suppressed constant depends on $p,\om$, $\Psi$ and $\|f\|_{A^p_{\om,\Psi}}$.
The proof will be completed by combining \eqref{zz} and \eqref{zzz}.
\end{proof}

We are now ready to prove Theorem~\ref{Tgbounded}. Before giving the proof, we should notice that the proof of Proposition~\ref{cptseq} with minor modification implies that $T_g: A^p_{\om,\Psi}\to A^q_{\nu,\Phi}$ is compact if and
only if $\lim_{n\to\infty} \|T_g(f_n)\|_{L^q_{\mu,\Phi}}=0$  for any bounded sequence $\{f_n\}$ in $A^p_{\om,\Psi}$ that converges to 0
uniformly on compact subsets on $\D$.

\medskip

\Prf \emph{Theorem~\ref{Tgbounded}(i)}.
Assume first that \eqref{Tgcondition} holds. By Lemma~\ref{ThetaIncreasing}, Lemma~\ref{Growthf1}, \eqref{Tgcondition} and Lemma~\ref{GrowthPsi}, we derive that there exists a $\beta=\beta(\nu,q)$ such that
		\begin{equation}\label{Tgfgrowth}
		\begin{split}
    |T_g(f)(z)|
		&=\left|\int_{0}^{z}f(\z)g'(\z)\,d\z\right|
		\le\int_{0}^{z}|f(\z)g'(\z)|\,|d\z|\\
    &\lesssim \int_{0}^{|z|}\left(\frac{\lVert f \rVert^p_{A^p_{\om,\Psi}}+1}{\om(S(\z))\Psi\left(\frac{1}{1-|\z|}\right)}\right)^\frac{1}{p}
		\frac{\left(\Psi\left(\frac{1}{\om(S(\z))}\right)\om(S(\z))\right)^\frac{1}{p}}
		{(1-|\z|)\left(\Phi\left(\frac{1}{\om(S(\z))}\right)\nu(S(\z))\right)^\frac{1}{q}}\,|d\z|\\
    &\lesssim \frac{(\lVert f \rVert^p_{A^p_{\om,\Psi}}+1)^{\frac1p}}{(1-|z|)^{\beta}},\quad f\in A^p_{\om,\Psi},\quad z\in\D,
		\end{split}
		\end{equation}
where the suppressed constant depends on $p,q,\om,\Psi,\Phi$ and $\|f\|_{A^p_{\om,\Psi}}$. This estimate enables us to prove that for any $\nu\in\DDD$ and $\Phi\in\LL$, we have
\begin{equation}\label{qw}
\|T_g(f)\|^q_{A^q_{\nu,\Phi}}\lesssim  1 + \int_{\D}|T_g(f)(z)|^q\Phi\left(\frac{1}{1-|z|}\right)\nu(z)\,dA(z),\quad g\in \H(\D).
\end{equation}
Indeed, if $\Phi$ is essentially increasing, then \eqref{Tgfgrowth} and Proposition~\ref{properties}(ii) yields \eqref{qw} directly. If $\Phi$ is essentially decreasing, define $E_{\veps}=\{z\in \D: |T_g(f)(z)|<\frac{1}{(1-|z|)^{\veps/q}}\}$, where $\veps=\veps(q,\nu)>0$ will be fixed in such a way that
$\int_{\D}\nu_{[-\veps]}(z)\,dA(z)<\infty.$
This estimate together with Proposition~\ref{properties}(ii) yields
\begin{equation*}
\begin{split}
        \lVert T_g(f) \rVert_{A^q_{\nu,\Phi}}^q
        &=\int_{\D}|T_g(f)(z)|^q\Phi\left(|T_g(f)(z)|\right)\nu(z)\,dA(z)\\
        &=\int_{E_{\veps}}|T_g(f)(z)|^q\Phi\left(|T_g(f)(z)|\right)\nu(z)\,dA(z)\\
        &\quad +\int_{\D \setminus E_{\veps}}|T_g(f)(z)|^q\Phi\left(|T_g(f)(z)|\right)\nu(z)\,dA(z)\\
        &\lesssim 1 + \int_{\D}|T_g(f)(z)|^q\Phi\left(\frac{1}{1-|z|}\right)\nu(z)\,dA(z),
\end{split}
\end{equation*}
which proves \eqref{qw}. An integration by parts together with Lemma~\ref{LemmaPsiOm} yields
\begin{equation}\label{qwe}
   \int_{\D}|T_g(f)(z)|^q\Phi\left(\frac{1}{1-|z|}\right)\nu(z)\,dA(z) \asymp  \int_{\D}|T_g(f)(z)|^q\Phi\left(\frac{1}{1-|z|}\right)\frac{\widehat{\nu}(z)}{1-|z|}\,dA(z).
\end{equation}
Since $\nu \in \DDD$, standard arguments imply $\frac{\widehat{\nu}}{1-|\cdot|} \in \DDD$, and hence $\Phi\left(\frac{1}{1-|\cdot|}\right)\frac{\widehat{\nu}}{1-|\cdot|} \in \DDD$ by Proposition \ref{setequiv}.  Then, applying the general Littlewood-Paley theorem (see \cite[Theorem 5]{PelRat2020}), \eqref{qwe} and \eqref{Tgcondition}, we get
\begin{equation}\label{1111}
\begin{split}
   & \quad\int_{\D}|T_g(f)(z)|^q\Phi\left(\frac{1}{1-|z|}\right)\nu(z)\,dA(z)\\
    &\asymp  \int_{\D}|T_g(f)(z)|^q\Phi\left(\frac{1}{1-|z|}\right)\frac{\widehat{\nu}(z)}{1-|z|}\,dA(z)\\
    &\asymp \int_{\D}|f(z)g'(z)|^q(1-|z|)^q\Phi\left(\frac{1}{1-|z|}\right)\frac{\widehat{\nu}(z)}{1-|z|}\,dA(z)\\
    &\lesssim \int_{\D}|f(z)|^q\left(\Psi\left(\frac{1}{1-|z|}\right)\right)^{\frac{q}{p}}\frac{(\om(S(z)))^\frac{q}{p}}{(1-|z|)^2}\,dA(z)\\
    &=\int_{\D}|f(z)|^p|f(z)|^{q-p}\left(\Psi\left(\frac{1}{1-|z|}\right)\right)^{\frac{q}{p}}\frac{(\om(S(z)))^\frac{q}{p}}{(1-|z|)^2}\,dA(z).
    \end{split}
\end{equation}
Applying the growth estimate of $|f|^{q-p}$ based on Lemma \ref{Growthf1} into the above estimate and \eqref{qw}, we derive
$$
\lVert T_g(f) \rVert_{A^q_{\nu,\Phi}}^q \lesssim 1 + \int_{\D}|f(z)|^p\Psi\left(\frac{1}{1-|z|}\right)\frac{\om(S(z))}{(1-|z|)^2}\,dA(z).
$$
From here we follow the same steps as were done in the proof of the sufficiency of Theorem~\ref{mainresult}(i).

Assume next that $T_g: A^p_{\om,\Psi} \to A^q_{\nu,\Phi}$ is bounded. By the subharmonicity of $|g'|^q$ and Proposition~\ref{properties}(ii), we have
\begin{equation*}
    \begin{split}
        &\quad \frac{|g'(a)|^q(1-|a|)^q\left(\Phi\left(\frac{1}{\om(S(a))}\right)\nu(S(a))\right)}
    {\left(\Psi\left(\frac{1}{\om(S(a))}\right)\om(S(a))\right)^\frac{q}{p}}\\
    &\lesssim \frac{(1-|a|)^{q-2}\Phi\left(\frac{1}{\om(S(a))}\right)\nu(S(a))}
    {\left(\Psi\left(\frac{1}{\om(S(a))}\right)\om(S(a))\right)^\frac{q}{p}}\int_{\Delta(a,r)}|g'(z)|^q\,dA(z)\\
    &\asymp \int_{\Delta(a,r)}|g'(z)|^q \frac{(1-|z|)^{q-2}\Phi\left(\frac{1}{\om(S(z))}\right)\nu(S(z))}{\left(\Psi\left(\frac{1}{\om(S(a))}\right)\om(S(a))\right)^\frac{q}{p}}\,dA(z)\\
    &\asymp \int_{\Delta(a,r)}\left(\frac{1-|a|}{|1-\overline{a}z|}\right)^{\frac{\gamma q}{p}}|g'(z)|^q \frac{(1-|z|)^{q-2}\Phi\left(\frac{1}{\om(S(z))}\right)\nu(S(z))}{\left(\Psi\left(\frac{1}{\om(S(a))}\right)\om(S(a))\right)^\frac{q}{p}}\,dA(z)\\
    &\lesssim \int_{\D}|f_a(z)|^q|g'(z)|^q (1-|z|)^{q}\Phi\left(\frac{1}{1-|z|}\right)\frac{\widehat{\nu}(z)}{1-|z|}\,dA(z),\quad a\in\D,
    \end{split}
\end{equation*}
where $f_a$ is exactly the same test function defined in Lemma~\ref{fa in Ap}. By similar reasoning as was done on the proof sufficiency, it follows from the afternote of Lemma \ref{LemmaPsiOm}, \cite[Theorem 5]{PelRat2020} and an integration by parts that
\begin{equation*}
    \begin{split}
        &\int_{\D}|f_a(z)|^q|g'(z)|^q (1-|z|)^{q}\Phi\left(\frac{1}{1-|z|}\right)\frac{\widehat{\nu}(z)}{1-|z|}\,dA(z)\\
        &\asymp \int_{\D}|T_g(f_a)(z)|^q \Phi\left(\frac{1}{1-|z|}\right)\frac{\widehat{\nu}(z)}{1-|z|}\,dA(z)\\
        &\lesssim
        \int_{\D}|T_g(f_a)(z)|^q \Phi\left(\frac{1}{1-|z|}\right)\nu(z)\,dA(z).
    \end{split}
\end{equation*}
Note that the omitted constant depends on $\|T_g(f_a)\|_{A^q_{\nu,\Phi}}$.  As $\| f_a \|_{A^p_{\om,\Psi}}$ is uniformly bounded, we can use the assumption that $T_g: A^p_{\om,\Psi} \to A^q_{\nu,\Phi}$ is bounded together with Proposition \ref{setequiv} to show that the constant is actually universal. By following the steps used in Cases I and II in the proof of Theorem~\ref{mainresult}, we finish the proof.

\medskip

\Prf \emph{Theorem~\ref{Tgbounded}(ii)}
Let $g\in\H(\D)$ and assume $T_g$ is compact. On one hand, since $f_a$ is uniformly bounded in $A^p_{\om,\Psi}$ and $f_a\to 0$ as $|a|\to1^-$, we have by Proposition~\ref{cptseq},
\begin{equation}\label{zx}
    \lim_{|a|\to1^-}\|T_g(f_a)\|_{A^q_{\nu,\Phi}}=0,\quad g\in\H(\D).
\end{equation}
On the other hand, a straightforward consequence from the proof of the necessity of (i) gives us
\begin{equation}\label{zx1}
\begin{split}
    \frac{|g'(a)|^q(1-|a|)^q\left(\Phi\left(\frac{1}{\om(S(a))}\right)\nu(S(a))\right)}
    {\left(\Psi\left(\frac{1}{\om(S(a))}\right)\om(S(a))\right)^\frac{q}{p}}&\lesssim
        \int_{\D}|T_g(f_a)(z)|^q \Phi\left(\frac{1}{1-|z|}\right)\nu(z)\,dA(z),
    \end{split}
\end{equation}
for $a\in\D$. The proof is complete by Proposition \ref{setequiv}
and \eqref{zx}.

Conversely, assume \eqref{Tgcondition2} holds. Let $\{f_n\}$ be a norm bounded sequence in $A^p_{\om,\Psi}$ such that $\|f_n\|_{A^p_{\om,\Psi}}<M$ and $f_n\to0$ uniformly on compact subsets of $\D$, as $n\to\infty$. To prove the assertion, by Proposition~\ref{cptseq}, it  suffices to show that
\begin{equation}\label{qaz}
\lim_{n\to\infty}\|T_g(f_n)\|_{A^q_{\mu,\Phi}}=0.
\end{equation}
To this end, let $\epsilon>0$. Then, the same proof as was used to verify \eqref{qw} together with the dominated convergence theorem implies that there exists a $N_1=N_1(\varepsilon)\in\N$ such that for all $n>N_1$,
\begin{equation*}
\|T_g(f_n)\|^q_{A^q_{\nu,\Phi}}\lesssim  \epsilon + \int_{\D}|T_g(f_n)(z)|^q\Phi\left(\frac{1}{1-|z|}\right)\nu(z)\,dA(z),\quad g\in \H(\D),
\end{equation*}
where the suppressed constant depends on $p,q,\om,\Psi,\Phi$ and $M$. Hence \eqref{1111} indicates that for any $R\in(0,1)$
\begin{equation}\label{sos}
\begin{split}
    \lVert T_g(f_n) \rVert_{A^q_{\nu,\Phi}}^q &\lesssim \epsilon + \int_{\D}|T_g(f_n)(z)|^q\Phi\left(\frac{1}{1-|z|}\right)\nu(z)\,dA(z)\\
    &\asymp\epsilon +\int_{\D}|f_n(z)g'(z)|^q(1-|z|)^q\Phi\left(\frac{1}{1-|z|}\right)\frac{\widehat{\nu}(z)}{1-|z|}\,dA(z)\\
    &=\epsilon +\left(\int_{D(0,R)}+\int_{\D\setminus D(0,R)}\right)|f_n(z)g'(z)|^q(1-|z|)^q\Phi\left(\frac{1}{1-|z|}\right)\frac{\widehat{\nu}(z)}{1-|z|}\,dA(z)\\
    &=\epsilon+I_1+I_2.
    \end{split}
\end{equation}
 Since \eqref{Tgcondition2} holds by the hypothesis, for the above $\epsilon$, there exists a $R=R(\epsilon)\in(0,1)$ such that for any $R<|a|<1$
\begin{equation}\label{111}
    \frac{|g'(a)|^q(1-|a|)^q\left(\Phi\left(\frac{1}{\om(S(a))}\right)\nu(S(a))\right)}
    {\left(\Psi\left(\frac{1}{\om(S(a))}\right)\om(S(a))\right)^\frac{q}{p}}<\epsilon.
\end{equation}
Therefore, for the above fixed $R$ and $\epsilon$, the dominated convergence theorem implies the existence of a natural number $N_2=N_2(R,\epsilon)>N_1$ such that for all $n>N_2$, the inequality $I_1\lesssim\epsilon$ holds. Moreover, by \eqref{111} and Lemma~\ref{Growthf}, we deduce
\begin{equation*}
    I_2\lesssim \epsilon\int_{\D}|f_n(z)|^p\Psi\left(\frac{1}{1-|z|}\right)\frac{\om(S(z))}{(1-|z|)^2}\,dA(z),
\end{equation*}
and therefore $I_2\lesssim\epsilon$ by the proof of the sufficiency of
Theorem~\ref{mainresult}
 and the fact that $\|f_n\|_{A^p_{\om,\Psi}}<M$ for all $n\in\N$. Applying the estimates of $I_1$ and $I_2$ into \eqref{sos}, we prove \eqref{qaz}. The proof is now complete.

\end{document}